\documentclass[12pt,reqno]{amsart}
\pdfoutput=1
\usepackage{latexsym, amsmath, amssymb, bm, dsfont}
\usepackage[colorlinks=true,linkcolor=blue,
            urlcolor=blue,citecolor=blue]{hyperref}
\usepackage[shortalphabetic]{amsrefs}
\usepackage[T1]{fontenc}
\usepackage{mathptmx}
\usepackage{microtype}
\usepackage{tikz}
\usetikzlibrary{arrows,shapes,positioning,calc,patterns}
\usepackage{graphicx}
\usepackage{caption}
\usepackage{subcaption}

\usepackage[centering, includeheadfoot, hmargin=1.2in, vmargin=0.8in,
  headheight=30.4pt]{geometry}
\newtheorem{lemma}{Lemma}[section]
\newtheorem{theorem}[lemma]{Theorem}
\newtheorem{corollary}[lemma]{Corollary}
\newtheorem{proposition}[lemma]{Proposition}
\theoremstyle{definition}
\newtheorem{remark}[lemma]{Remark}
\newtheorem{definition}[lemma]{Definition}
\newtheorem{example}[lemma]{Example}
\theoremstyle{remark}
\newcommand{\define}[1]{{\bfseries\itshape #1}}
\renewcommand{\theequation}%
{\arabic{section}.\arabic{lemma}.\arabic{equation}}

\renewcommand{\geq}{\geqslant}
\renewcommand{\leq}{\leqslant}
\newcommand{\relphantom}[1]{\mathrel{\phantom{#1}}}
\newcommand{\kk}{\ensuremath{\Bbbk}} 
 
\newcommand{\NN}{\ensuremath{\mathbb{N}}} 
 
\newcommand{\QQ}{\ensuremath{\mathbb{Q}}} 
 
\newcommand{\ZZ}{\ensuremath{\mathbb{Z}}} 

\DeclareMathOperator{\reg}{reg}
\DeclareMathOperator{\Tor}{Tor}
\DeclareMathOperator{\convergent}{\textsf{\upshape c}}
\begin{document}

\title{Cones of Hilbert Functions}

\author[M.~Boij]{Mats Boij}
\address{Department of Mathematics\\ Royal Institute of Technology (KTH)\\
  Stockholm\\ 100 44\\ Sweden} \email{boij@kth.se}

\author[G.G.~Smith]{Gregory G. Smith} 
\address{Department of Mathematics and Statistics \\ Queen's University \\
  Kingston \\ Ontario\\ K7L~3N6\\ Canada}
\email{ggsmith@mast.queensu.ca}


\begin{abstract}
  We study the closed convex hull of various collections of Hilbert functions.
  Working over a standard graded polynomial ring with modules that are
  generated in degree zero, we describe the supporting hyperplanes and extreme
  rays for the cones generated by the Hilbert functions of all modules, all
  modules with bounded $a$-invariant, and all modules with bounded
  Castelnuovo--Mumford regularity.  The first of these cones is
  infinite-dimensional and simplicial, the second is finite-dimensional but
  neither simplicial nor polyhedral, and the third is finite-dimensional and
  simplicial.
\end{abstract}

\maketitle

\vspace*{-1em}

\section{Introduction} 
\label{sec:intro}

\noindent
Classifying modules is a universal problem in algebra.  Within commutative
algebra, the classification of graded modules bifurcates into understanding
the space of all modules with a specified Hilbert function and describing the
numerical functions that arise as the Hilbert function of some module.  As a
counterpart to multigraded Quot schemes which parametrize the modules with a
fixed Hilbert function (see \cite{HS}*{\S6.2}), this paper initiates the study
of the closed convex cones generated by the Hilbert functions of a given
collection of modules.

There are many compelling collections of modules to consider over the standard
graded ring $S := \kk[x_0,x_1, \dotsc, x_n]$ where $\kk$ is a field.  The most
naive consists of all finitely generated $\NN$-graded $S$-modules.  In this
case, every point in the corresponding closed convex cone is a unique
countable linear combination of the Hilbert functions of the $S$-modules
$S(-i)/\langle x_0, x_1, \dotsc, x_n\rangle \cong \kk(-i)$ for $i \in \NN$.
Hence, within any relevant topological vector space, the closed convex hull is
the simplicial cone generated by the Hilbert functions of these artinian
modules.  In other words, it is simply the infinite-dimensional positive
orthant.  To actually capture the subtleties of homogeneous coordinate rings,
we concentrate on collections of $\NN$-graded $S$-modules that are generated
in degree zero.  If $E \subseteq \QQ^\NN$ is a topological $\QQ$-vector space
that contains the Hilbert functions of all artinian $S$-modules generated in
degree zero and the function $h \colon \NN \to \QQ$ lies in $E$, then our
first substantive result is the following.

\begin{theorem}
  \label{thm:one}
  The closed convex hull of the Hilbert functions of $S$-modules generated in
  degree zero and contained in $E$ is the intersection of the closed
  half-spaces defined by the inequalities $(n+j+1)h(j) \geq (j+1)h(j+1)$ for
  $j \in \NN$.  The extreme rays of this simplicial cone are generated by the
  Hilbert functions of the $S$-modules $S/\langle x_0, x_1, \dotsc, x_n
  \rangle^i$ where $i \in \NN$.
\end{theorem}

By design, our approach overcomes limitations in Macaulay's celebrated theorem
on Hilbert functions.  Although the main theorem in \cite{Macaulay} determines
those numerical functions which occur as Hilbert functions of a homogeneous
quotient of $S$, the complexity of this result, as underscored in
\cite{BFS}*{p.~27} and \cite{Brenti}*{p.~132}, makes it unwieldy.  The optimal
linear conditions are frequently more useful despite not providing a complete
characterization.  Moreover, because Macaulay's Theorem depends inherently on
lex-segment ideals, it cannot be extended to graded rings that do not have
analogous ideals.  Closed convex hulls enjoy no such restrictions.  These two
features, in addition to the advantages of endowing the set of Hilbert
functions with a geometric structure, motivate our interest in cones of
Hilbert functions.  In particular, we regard the supporting hyperplanes in
Theorem~\ref{thm:one} (see Theorem~\ref{thm:Macaulay}) as the linearization of
Macaulay's Theorem.

To reveal the properties related to Hilbert polynomials, we need a smaller
collection of modules---one that does not contain artinian modules of
arbitrary length.  Requiring that the Hilbert polynomial and Hilbert function
agree for all integers greater than a fixed number $a$ is a straightforward
method of making such a collection.  Equivalently, we restrict to the
finite-dimensional $\QQ$-vector space $V_{n,a} \subset \QQ^\NN$ consisting of
all functions $h \colon \NN \to \QQ$ satisfying $\sum_{j \in \NN} h(j) t^j =
(b_0 + b_1 t + \dotsb + b_{a+n} t^{a+n})/(1-t)^{n}$ for some $b_0, b_1,
\dotsc, b_{a+n} \in \QQ$.  In this context, the primary object of interest is
the closed convex hull $Q_{n,a} \subset V_{n,a}$ of the Hilbert functions of
finitely generated $\NN$-graded $S$-modules that are generated in degree zero
and have no free summands.  Our second major result characterizes this cone.

\begin{theorem}
  \label{thm:two}
  If $T \colon \QQ^\NN \to \QQ^\NN$ is the linear operator defined by 
  \[
  \bigl( T[h] \big)(j) := (n+j+1)h(j) - (j+1) h(j+1) 
  \]
  where $h \colon \NN \to \QQ$, then the image $T[Q_{n,a}]$ equals the closed
  convex hull of $\NN^\NN \cap V_{n,a}$.
\end{theorem}

\noindent
The linear operator $T$ and the supporting hyperplanes for $Q_{n,a}$ arise
from the linearization of Macaulay's Theorem.  Since
Proposition~\ref{pro:Pcone} describes the extreme rays for the image
$T[Q_{n,a}]$, we also obtain, in Corollary~\ref{cor:Qcone}, a description for
the extreme rays of $Q_{n,a}$.  As Example~\ref{exa:Q_{3,-1}} demonstrates,
the cone $Q_{n,a}$ is generally neither simplicial nor polyhedral.

Alternatively, Castelnuovo--Mumford regularity, an invariant of a module not
just its Hilbert function, provides a more sophisticated mechanism for
creating a smaller collection.  To be explicit, let $R_{n,m}$ be the closed
convex hull in $V_{n+1,m}$ of the Hilbert functions of finitely generated
$\NN$-graded $S$-modules that are generated in degree zero, have no free
summands, and have regularity at most $m$.  If $q_h \in \QQ[s]$ denotes the
Hilbert polynomial associated to $h \in R_{n,m}$ and $\nabla \colon \QQ[s] \to
\QQ[s]$ is the backward difference operator defined by $\nabla q(s) := q(s) -
q(s-1)$ for $q \in \QQ[s]$, then our third significant result describes the
cone $R_{n,m}$.

\begin{theorem}
  \label{thm:three}
  The closed convex cone $R_{n,m}$ lies in the subspace $V_{n,m} \subset
  V_{n+1,m}$ and is the intersection of the closed half-spaces given by the
  inequalities:
  \begin{xalignat*}{2}
    (n+j+1) h(j) &\geq (j+1) h(j+1) && \text{for $0 \leq j < m$,} \\
    h(m) &\geq q_h(m) & & \text{and,}\\
    (n+1-i) \nabla^i q_h (m) &\geq (n+m+1-i) \nabla^{i+1} q_h (m) && \text{for
      $0 \leq i < n$.}
  \end{xalignat*}
  The extreme rays of this simplicial polyhedral cone are generated by the
  Hilbert functions of the following cyclic modules:
  \begin{equation}
    \label{eq:cyclic}
    \begin{split}
      \frac{S}{\langle x_0, x_1, \dotsc, x_{n} \rangle}, &\frac{S}{\langle x_0,
        x_1, \dotsc, x_{n} \rangle^2}, \dotsc, \frac{S}{\langle x_0, x_1,
        \dotsc, x_{n} \rangle^{m}}, \\
      &\frac{S}{\langle x_0, x_1, \dotsc, x_{n} \rangle^{m+1}},
      \frac{S}{\langle x_0, x_1, \dotsc, x_{n-1} \rangle^{m+1}}, \dotsc,
      \frac{S}{\langle x_0 \rangle^{m+1}} \, .
    \end{split}
  \end{equation}
\end{theorem}

\noindent
To prove this theorem, we use the natural projection from the cone of Betti
tables.  It is intriguing that the extreme rays of $R_{n,m}$ correspond to
modules with linear free resolutions, arguably the simplest pure Betti tables.

Our progress in describing cones of Hilbert functions points in several
promising directions.  For instance, how does one describe the closed convex
hull for other important collections of $S$-modules.  Since convex cones are
closed under linear combinations with positive coefficients and the Hilbert
function of a direct sum is the sum of the Hilbert functions, collections of
modules that are closed under finite direct sums are likely the most
pertinent.  In contrast, we would also like to generalize Macaulay's Theorem
to other rings by describing the closed convex hull of the Hilbert functions
of all module generated in degree zero.  Following \cite{GHP}, toric rings are
the most prominent candidates among $\NN$-graded commutative rings.  More
generally, what is the analogue of Theorem~\ref{thm:one} when $S$ is replaced
by the homogeneous coordinate ring of a projective variety and how do the
supporting hyperplanes and extreme rays reflect the geometry of the underlying
variety.  Considering non-standard and multigraded polynomial rings branches
onto a somewhat different track as \cite{BM}, \cite{BNV}, and \cite{KMU}
establish.  Preliminary work for a standard bigraded polynomial ring, or
equivalently the Cox ring for a product of projective spaces, indicates that
an elementary variant of Theorem~\ref{thm:one} holds.  However, versions over
the Cox ring for any smooth projective toric variety appear to be
intrinsically more complicated.  For geometric applications, one should
probably exclude all modules that contain an element annihilated by a power of
the irrelevant ideal.  Finally, we have not begun to analyze the semigroup
within the closed convex cone formed by the Hilbert functions of modules.

\subsection*{Contents of the paper}
Section~\ref{sec:general} gives both a combinatorial proof and an algebraic
proof for the linearization of Macaulay's Theorem, also known as
Theorem~\ref{thm:Macaulay}.  Our description of the closed convex hull of the
Hilbert function of artinian $S$-modules generated in degree zero, given in
Corollary~\ref{cor:TVS}, and the proof for Theorem~\ref{thm:one} follow.  In
Section~\ref{sec:a-invariant}, Proposition~\ref{pro:Pcone} describes the
extreme rays of the closed convex hull of $\NN^\NN \cap V_{n,a}$.  After a
triple of technical lemmata, we prove Theorem~\ref{thm:two}.  The section ends
with Corollary~\ref{cor:Qcone}, which explicitly describes the supporting
hyperplanes and extreme rays of $Q_{n,a}$, and four examples illustrating this
corollary.  We prove Theorem~\ref{thm:three} in Section~\ref{sec:regularity}
and we close with Proposition~\ref{pro:BettiBounds}, which explicitly bounds
Betti numbers linearly via Hilbert functions.

\subsection*{Conventions} 
We write $\NN$ for the set of non-negative integers and $\kk$ for an arbitrary
field.  A set is countable if it has the same cardinality as $\NN$.
Throughout the document, the polynomial ring $S := \kk[x_0, x_1, \dotsc, x_n]$
has the standard $\NN$-grading induced by setting $\deg(x_i) = 1$ for all $0
\leq i \leq n$.  All $S$-modules are finitely generated and $\NN$-graded.

\subsection*{Acknowledgements}
We thank Kathrin Vorwerk and Mike Roth for their valuable insights.  The
computer software \emph{Macaulay2}~\cite{M2} was indispensable for generating
examples and discovering the correct statements.  The second author was
partially supported by NSERC.

\section{Modules generated in degree zero}
\label{sec:general}

\noindent
This section considers the closed convex hull of Hilbert functions of
$S$-modules generated in degree zero.  The key result, namely
Theorem~\ref{thm:Macaulay}, describes the linear inequalities satisfied by the
Hilbert function of such a module.  By working in appropriate
infinite-dimensional topological vector spaces, we obtain descriptions of the
supporting hyperplanes and the extreme rays for the closed convex hull of
Hilbert functions for any collection of modules containing all artinian
$S$-modules.

If $M$ is a finitely generated $\NN$-graded $S$-module, then its Hilbert
function is the numerical function $h_M \colon \NN \to \NN$ defined by
$h_M(j) := \dim_\kk M_j$.

\begin{theorem}
  \label{thm:Macaulay}
  The Hilbert function of a finitely generated $\NN$-graded $S$-module $M$
  generated in degree zero satisfies the following inequalities for all $j \in
  \NN$:
  \begin{xalignat*}{3}
    \frac{h_M(j)}{h_S(j)} &\geq \frac{h_M(j+1)}{h_S(j+1)} & &\text{or}
    & (n+j+1) \, h_M(j) &\geq (j+1) h_M(j+1) \, .
  \end{xalignat*}
\end{theorem}

\begin{proof}[Combinatorial Proof]
  We have $h_S(j) = \binom{n+j}{j}$ for all $j \in \NN$ and the Absorption
  Identity states that $k \binom{\ell}{k} = \ell \binom{\ell-1}{k-1}$ for all
  $k,\ell \in \ZZ$, so the two forms of inequalities are equivalent and it is
  enough to prove that
  \begin{align}
    \label{eq:Macaulay}
    (n+j+1) \, h_M(j) \geq (j+1) h_M(j+1) \, .
  \end{align}
  Since $M$ is generated in degree zero, there exists a surjective
  homomorphism of $\NN$-graded $S$\nobreakdash-modules $\eta \colon S^{(m)}
  \to M$ where $S^{(m)}$ is the $m$-fold direct sum of $S$ for some $m \in
  \NN$.  By choosing a monomial order on $S^{(m)}$, we see that both $M$ and
  the quotient of $S^{(m)}$ by the monomial submodule generated by the leading
  terms of $\ker(\eta)$ have the same Hilbert function.  In both cases, the
  monomials not belonging to the initial submodule form a $\kk$-vector spaces
  basis; see Theorem~15.3 in \cite{Eisenbud}.  Hence, it suffices to establish
  the inequality \eqref{eq:Macaulay} in the case $M = S/I$ for some monomial
  ideal $I$.

  To accomplish this, we interpret both sides of the inequality
  \eqref{eq:Macaulay} as cardinalities of sets and describe an appropriate
  injective map.  Using the stars-and-bars correspondence (see \S1.2 in
  \cite{Stanley}), we identify the set $\mathcal{M}_{j+1}$ of monomials in
  $S_{j+1}$ with the $(j+1)$-subsets of $\{1,2, \dotsc, n+j+1\}$.  Consider
  the set $\mathcal{X} \subseteq \{1, 2, \dotsc, j+1\} \times
  \mathcal{M}_{j+1}$ consisting of all pairs $(i,\sigma)$ such that $i \in
  \sigma$, and let $\mathcal{Y} := \{1, 2, \dotsc, n+j+1\} \times
  \mathcal{M}_j$.  Define the map $\Phi \colon \mathcal{X} \to \mathcal{Y}$ by
  $\Phi(i,\sigma) = (i, \sigma \setminus \{i\})$.  This map is injective,
  because we can reconstruct $\sigma$ from the pair $(i, \sigma \setminus
  \{i\})$.  If $\mathcal{X}' \subseteq \mathcal{X}$ and $\mathcal{Y}'
  \subseteq \mathcal{Y}$ are the subsets for which the second components
  correspond to monomials not in $I$, then we have 
  \begin{xalignat*}{3}
    |\mathcal{X}'| &= (j+1) h_{S/I}(j+1) && \text{and} & |\mathcal{Y}'| &=
    (n+j+1) h_{S/I}(j) \, .
  \end{xalignat*}
  Since $\sigma \setminus \{j\}$ corresponds to a monomial not in $I$ whenever
  $\sigma$ corresponds to a monomial not in $I$, restricting the map $\Phi$
  yields the required injection from $\mathcal{X}'$ to $\mathcal{Y}'$.
\end{proof}

\begin{proof}[Algebraic Proof]
  Generalizing Macaulay's characterization of Hilbert functions of
  $\NN$-graded $\kk$-algebras (i.e.\ the Main Theorem in \cite{Macaulay} or
  Theorem~4.2.10 in \cite{BH}), Corollary~6 in \cite{Hulett} implies that the
  Hilbert function of $M$ is bounded above by the Hilbert function of the
  quotient of a free module $S^{(m)}$ by a lexicographic submodule.  In
  particular, if $h_M(j)$ is a multiple of $h_S(j) = \binom{n+j}{j}$, then
  $h_M(j+1)$ is bounded above by the same multiple of $h_S(j+1) =
  \binom{n+j+1}{j+1} = \frac{n+j+1}{j+1} \binom{n+j}{j} = \frac{n+j+1}{j+1}
  h_S(j)$.  Hence, for an appropriate $k \in \NN$, $h_{M^{(k)}}(j)$ is a
  multiple of $h_S(j)$ and we obtain
  \begin{align*}
    k \, h_M(j+1) &= h_{M^{(k)}}(j+1) \leq \left( \frac{n+j+1}{j+1} \right)
    h_{M^{(k)}}(j) = k \, \left( \frac{n+j+1}{j+1} \right) h_{M}(j) \,
    . \qedhere
  \end{align*}
\end{proof}

\begin{remark}
  If one replaces the symmetric algebra $S$ with an exterior algebra (see
  Corollary~4.18 in \cite{BNV}), then the analogue of
  Theorem~\ref{thm:Macaulay} also holds.  However, these inequalities do not
  hold in all rings.  For example if $R := \kk[x_0,x_1]/\langle x_0^2, x_0 x_1
  \rangle$ and $M := R/\langle x_0 \rangle \cong \kk[x_1]$, then we have
  $h_M(1)/h_R(1) = \frac{1}{2} < 1 = h_M(2)/h_R(2)$.
\end{remark}

Let $\convergent_0 \subset \QQ^\NN$ be the Banach space consisting of all
convergent real sequences $h \colon \NN \to \QQ$ such that $h(j) \to 0$ as $j
\to \infty$ equipped with the sup norm; see \cite{TVS}*{p.\ 31}.  For any
finitely generated $\NN$-graded artinian $S$-module $M$, we have $h_M \in
\convergent_0$, because the sequence is eventually zero.

\begin{corollary}
  \label{cor:TVS}
  The closed convex hull in $\convergent_0$ of the Hilbert functions of
  artinian $S$-modules generated in degree zero is the intersection of the
  closed half-spaces defined by the inequalities $(n+j+1) \, h(j) \geq (j+1)
  \, h(j+1)$ for $j \in \NN$.  Moreover, the extreme rays of this cone are
  generated by the Hilbert functions of the $S$-modules $S/\langle x_0, x_1,
  \dotsc, x_n \rangle^i$ where $i \in \NN$.
\end{corollary}

\begin{proof}
  Theorem~\ref{thm:Macaulay} shows that the Hilbert function of any $S$-module
  $M$ generated in degree zero is contained in the intersection of the closed
  half-spaces determined by the inequalities $(n+j+1) \, h_M(j) \geq (j+1)
  \, h_M(j+1)$ for all $j \in \NN$.  For brevity, set $\mathfrak{m} :=
  \langle x_0, x_1, \dotsc, x_n \rangle$.  For $i \in \NN$, we have
  \begin{align}
    \label{eq:artinian}
    h_{S/\mathfrak{m}^i}(j) = 
    \begin{cases}
      \binom{n+j}{j} & \text{if $j < i$} \\
      0 & \text{if $j \geq i$.}
    \end{cases}
  \end{align}
  These Hilbert functions are linearly independent in $\convergent_0$, so each
  $S/\mathfrak{m}^i$ corresponds to an extreme ray of the closed convex cone
  $K$ generated by $\{ \bigl( h_{S / \mathfrak{m}^i}(j) \bigr) : i \in \NN
  \}$.  Equation~\eqref{eq:artinian} also yields $(n+i) \,
  h_{S/\mathfrak{m}^i}(i-1) > (i) \, h_{S/\mathfrak{m}^i}(i)$ and,
  together with the Absorption Identity, shows that $(n+j+1) \,
  h_{S/\mathfrak{m}^i}(j) = (j+1) \, h_{S/\mathfrak{m}^i}(j+1)$ for all $j
  \neq i-1$.  Since $\frac{1}{\ell!}\binom{n+\ell-1}{n} \to 0$ as $\ell \to
  \infty$, the sequences
  \[
  \sum_{
    \begin{subarray}{c}
      k = 0\\
      k \neq i
    \end{subarray}
  }^{\ell} \frac{1}{k!} h_{S / \mathfrak{m}^k}(j)
  \]
  converge as $\ell \to \infty$ and the limit lies on the closed hyperplane
  $(n+j+1) \, h(j) = (j+1) \, h(j+1)$ if and only if $j = i-1$.  Hence, $K$ is
  intersection of the closed half-spaces defined by the inequalities $(n+j+1)
  \, h(j) \geq (j+1) \, h(j+1)$ for $j \in \NN$.  Finally, the Krein-Milman
  Theorem (e.g. Theorem~1 on \cite{TVS}*{p.\ 187} and the Corollary on
  \cite{TVS}*{p.\ 189}) establishes that every extreme ray of the cone of
  Hilbert functions for artinian $S$-modules generated in degree zero
  corresponds to an $S$-module $S/\mathfrak{m}^i$ for some $i \in \NN$.
\end{proof}

\begin{remark}
  The proof of Corollary~\ref{cor:TVS} exploits only the topological vector
  space structure of the Banach space $\convergent_0$.
\end{remark}

\begin{remark}
  Since every Cohen-Macaulay module has an artinian reduction,
  Corollary~\ref{cor:TVS} leads immediately to a description of the closed
  convex hull of the Hilbert function of Cohen-Macaulay $S$-modules generated
  in degree zero.
\end{remark}

By working in a larger space, we can extend Corollary~\ref{cor:TVS}.  Let $E
\subseteq \QQ^\NN$ be a topological $\QQ$-vector space that contains the
Hilbert functions of all artinian $S$-modules generated in degree zero.  For
example, if $E$ is the weighted $\ell^\infty$-space consisting of all bounded
sequences $h \colon \NN \to \QQ$ with respect to the norm $\| h \| := \sup_j
\bigl| n^{-j} h(j) \bigr|$, then the Hilbert function of every finitely
generated $\NN$-graded $S$-module is contained in $E$.

\begin{proof}[Proof of Theorem~\ref{thm:one}]
  Set $\mathfrak{m} := \langle x_0, x_1, \dotsc, x_n \rangle$.  Since every
  Hilbert function in $E$ can expressed uniquely as a non-negative countable
  linear combination of the Hilbert functions of the $S$-modules $S /
  \mathfrak{m}^i$ for $i \in \NN$, the cone of all Hilbert functions in $E$ is
  generated by $\{ h_{S/\mathfrak{m}^i}(j) : i \in \NN\}$.  Hence, the
  assertions follow from Corollary~\ref{cor:TVS}.
\end{proof}

\begin{remark}
  The closed convex cones described in Theorem~\ref{thm:one} and
  Corollary~\ref{cor:TVS} are both simplicial (in the sense of Choquet
  theory).  In other words, every point in the cone is a unique countable
  linear combination of the Hilbert functions of the $S$-modules $S/\langle
  x_0, x_1, \dotsc, x_n \rangle^i$ where $i \in \NN$.
\end{remark}

\section{Modules with bounded \texorpdfstring{$a$}{a}-invariant}
\label{sec:a-invariant}

\noindent
In this section, we replace the ambient infinite-dimensional vector space $E$
appearing in Section~\ref{sec:general} with a finite-dimensional vector space.
We accomplish this by concentrating on $S$\nobreakdash-modules with bounded
$a$-invariant.  In other words, we insist that the Hilbert function and
Hilbert polynomial agree for all integers greater than $a$.  To determine the
supporting hyperplanes and extreme rays, we related the cone of Hilbert
functions with bounded $a$\nobreakdash-invariant to the cone of non-negative
sequences.

Fix $n \in \NN$ and let $a \in \ZZ$ satisfy $a \geq -n$.  Consider the
finite-dimensional subspace $V_{n,a} \subset \QQ^\NN$ consisting of all
sequences $h \colon \NN \to \QQ$ such that the associated generating functions
satisfy
\[
\sum_{j \in \NN} h(j) t^j = \frac{b_0 + b_1 t + \dotsb + b_{a+n} t^{a+n}}
{(1-t)^n} \in \QQ(t)
\]
for some $b_0, b_1, \dotsc, b_{a+n} \in \QQ$.  Following Corollary~4.3.1 in
\cite{Stanley}, this condition on the generating function is equivalent to the
existence of $q_h \in \QQ[s]$ such that $q_h(j) = h(j)$ for all $j > a$.  As
an abuse of terminology, we refer to $q_h \in \QQ[s]$ as the Hilbert
polynomial of $h \in V_{n,a}$.  We identify a sequence $h \in V_{n,a}$ with
its generating function $\sum_{j} h(j) t^j \in \QQ(t)$ and regard $V_{n,a}$ as
a subspace of $\QQ(t)$.

\begin{definition}
  \label{def:positive}
  Let $P_{n,a}$ denote the closed convex hull in $V_{n,a}$ of the intersection
  $\NN^\NN \cap V_{n,a}$.  Informally, we say that $P_{n,a} \subset V_{n,a}$
  is the \define{cone of non-negative sequences}.
\end{definition}

\noindent
As $j \to \infty$, the inequalities $h(j) \geq 0$ for $j \in \NN$, which
define the cone $P_{n,a}$, simply assert that the leading coefficient of the
Hilbert polynomial is positive.

To give the dual description of $P_{n,a}$, it is convenient to introduce a
family of polynomials.  For an integer partition $\lambda := (\lambda_1,
\lambda_2, \dotsc, \lambda_r)$ where $\lambda_1 \geq \lambda_2 \geq \dotsb
\geq \lambda_r \geq 0$, we define
\[
p_{\lambda}(s) := \prod_{i = 1}^{r} (s - \lambda_{r-i+1} - 2i +2)(s -
\lambda_{r-i+1} -2i+1) \in \ZZ[s] \, .
\]  
Hence, the set $\{ p_{\lambda} : \lambda \vdash r\}$ consists of all monic
polynomials of degree $2r$ with non-negative integer roots that appear in
consecutive pairs.

The extreme rays in $P_{n,a}$ depends on the parameter $a$.  For instance,
there is an additional type of extreme ray when $a \geq 0$.  Nevertheless, we
provide a uniform characterization by introducing an auxiliary parameter
$\hat{a}$.

\begin{proposition}
  \label{pro:Pcone}
  Set $\hat{a} := a + \max(1,-a)$. The extreme rays of the non-negative cone
  $P_{n,a}$ correspond to the polynomials $1,t, \dotsc, t^{a}$ and the power
  series
  \begin{xalignat*}{2}
    &\sum_{j \geq \hat{a}} \Bigl( p_\lambda(j-\hat{a}) \prod_{\ell =
      1}^{\hat{a}-a-1} (j + \ell) \Bigr) t^j \, , & &\sum_{j \geq \hat{a}}
    \Bigl( p_\mu(j-\hat{a}-1) \prod_{\ell=0}^{\hat{a}-a-1} (j + \ell) \Bigr) t^j
  \end{xalignat*}
  where $\lambda$ ranges over all integer partitions with at most $\lfloor
  (n-\hat{a}+a)/2 \rfloor$ parts and $\mu$ ranges over all integer partitions
  with at most $\lfloor (n-\hat{a}+a-1)/2 \rfloor$ parts.
\end{proposition}

\begin{proof}
  The Binomial Theorem yields both $t^k = (1-t)^{-n} \sum\nolimits_{i}
  \binom{n}{i}(-1)^i t^{k+i}$ for $0 \leq k \leq a$ and $(1-t)^{-\ell} =
  (1-t)^{-n} \sum\nolimits_{i} \binom{n-\ell}{i}(-1)^i t^i$ for $1 \leq \ell
  \leq n$.  Hence, when $a \geq 0$, the rational functions $t^a, t^{a-1},
  \dotsc, 1, (1-t)^{-1}, (1-t)^{-2}, \dotsc, (1-t)^{-n}$ form a triangular
  basis for $V_{n,a}$.  Let $(c_{-a}, c_{-a+1}, \dotsc, c_0, c_{1}, c_{2},
  \dotsc, c_{n})$ denote the coordinates of $h \in P_{n,a}$ with respect to
  this ordered basis.  When $a < 0$, just the rational functions $(1-t)^{a},
  (1-t)^{a-1}, \dotsc, (1-t)^{-n}$ form a triangular basis for $V_{n,a}$.  For
  consistency, let $(c_{-a}, c_{-a+1}, \dotsc, c_{n})$ denote the coordinates
  of $h \in P_{n,a}$ in this situation.  Since the Generalized Binomial
  Theorem implies that $(1-t)^{-\ell} = \sum_{j} \binom{\ell+j-1}{\ell-1}
  t^j$, we obtain the inequalities
  \begin{xalignat*}{2}
    c_{-j} + \sum_{\ell=1}^{n} \binom{\ell+j-1}{\ell-1} c_{\ell} = h(j)
    &\geq 0 & & \text{for $0 \leq j \leq a$, and} \\
    \sum_{\ell = \hat{a} - a}^{n} \binom{\ell+j-1}{\ell-1} c_{\ell} = h(j) &\geq
    0 & &\text{for $j \geq \hat{a}$.}
  \end{xalignat*}
  The binomial coefficient $\binom{\ell+s-1}{\ell-1}$ is a polynomial in
  $\QQ[s]$ of degree $\ell-1$.  Hence, for $s \gg 0$, the leading coefficient
  of $\sum_\ell \binom{\ell+s-1}{\ell-1} c_\ell$ determines its sign and we
  obtain $c_n \geq 0$.
  
  Since $\dim(V_{n,a}) = n+a+1$, any extreme ray $h \in P_{n,a}$ must satisfy
  $h(j) = 0$ for at least $n+a$ distinct $j \in \NN$.  Suppose that we have at
  least $n$ equalities $h(j) = 0$ with $j \geq \hat{a}$.  It follows that
  $c_\ell = 0$ for $1 \leq \ell \leq n$, because the polynomials $\bigl\{
  \binom{\ell+s-1}{\ell-1} : 1 \leq \ell \leq n \bigr\}$ form a triangular
  basis for the vector space of all polynomials in $\QQ[s]$ with degree at
  most $n-1$.  To obtain a ray, we must also have $c_{-j} = h(j) = 0$ for all
  but one $j$ satisfying $0 \leq j \leq a$.  Hence, we have $a \geq 0$ and the
  extreme rays in this case correspond to the polynomials $1, t, \dotsc, t^a$.

  Now, suppose that $c_n > 0$ and that we have at most $n - \hat{a} + a$
  equalities $h(j) = 0$ with $j \geq \hat{a}$.  To obtain a ray, we must have
  $h(j) = 0$ for all $0 \leq j \leq a$ and $h(j) = 0$ for exactly $n - \hat{a}
  + a$ distinct $j$ satisfying $j \geq \hat{a}$.  Hence, the polynomial
  \[
  q(s) := \sum\limits_{\ell = \hat{a}-a}^{n} \binom{\ell+s+\hat{a}-1}{\ell-1}
  c_\ell \in \QQ[s]
  \] 
  has $n - \hat{a} + a$ distinct non-negative integer roots.  This polynomial
  also has $\hat{a} - a -1$ distinct negative integer roots, namely $-1, -2,
  \dotsc, 1 - \hat{a} + a$.  Since $\deg(q) = n-1$, it is uniquely determined
  by its leading coefficient and these integer roots.  Furthermore, the real
  function $q$ changes sign at each root and the evaluation of $q$ at every
  non-negative integer is non-negative, so the non-negative roots of $q$ must
  come in consecutive pairs.  When $n - \hat{a}+a$ is odd, we need an even
  number of sign changes arising from the non-negative roots, so $0$ itself
  must be a root of $q$.  Thus, the extreme rays in this case correspond to
  the power series
  \begin{xalignat*}{3}
    &\sum_{j \geq \hat{a}} \Bigl( p_\lambda(j-\hat{a}) \prod_{\ell =
      1}^{\hat{a}-a-1} (j + \ell) \Bigr) t^j & & \text{or} &\sum_{j \geq
      \hat{a}} \Bigl( p_\mu(j-\hat{a}-1) \prod_{\ell=0}^{\hat{a}-a-1} (j + \ell)
    \Bigr) t^j
  \end{xalignat*}
  where $n-\hat{a}+a$ is even and the integer partition $\lambda$ has
  $(n-\hat{a}+a)/2$ parts or $n-\hat{a}+a$ is odd and the integer partition
  $\mu$ has $(n - \hat{a} + a - 1)/2$ parts.

  The remaining extreme rays of $P_{n,a}$ lie in the hyperplane $c_n = 0$ or
  equivalently $V_{n-1,a}$.  Therefore, induction on $n$ completes the proof.
\end{proof}

The more important cone in $V_{n,a}$ is generated by Hilbert functions.
Specifically, if $M$ is any finitely generated $\NN$-graded $S$-module without
free summands (i.e.\ $\dim(M) < n+1$), then the Hilbert function $h_M \colon
\NN \to \QQ$ is contained in $V_{n,a}$ for all $a \gg 0$; see Corollary~4.1.8
in \cite{BH}.  Moreover, we have $h_M \in V_{n,a}$ if and only if the
Hilbert function $h_M(j)$ equals the Hilbert polynomial $q_M(j)$ for all $j
> a$; see Corollary~4.1.12 in \cite{BH}.  When $M$ is a $\kk$-algebra, the
parameter $a$ is the $a$-invariant; see Definition~4.4.4 in \cite{BH}.

\begin{definition}
  \label{def:a-invariant}
  Let $Q_{n,a}$ denote the closed convex hull in $V_{n,a}$ of the Hilbert
  functions of finitely generated $\NN$-graded $S$-modules that are generated
  in degree zero and have no free summands.  Informally, we say that $Q_{n,a}
  \subset V_{n,a}$ is the \define{cone of Hilbert functions with bounded
    $a$-invariant}.
\end{definition}

To encode the inequalities appearing in Theorem~\ref{thm:Macaulay}, we
introduce the linear operator $T \colon \QQ^\NN \to \QQ^\NN$ defined by
$\bigl(T[h]\bigr)(j) := (n+j+1) h(j) - (j+1) h(j+1)$ for any $h \colon \NN \to
\QQ$.  For an associated generating function, we have $T \left[ \sum_j h(j)
  t^j \right] = \sum_j \bigl( T[h](j) \bigr) t^j$.  Despite our notation, the
operator $T$ depends on the parameter $n$.  Our first lemma shows that the
restriction of $T$ to $V_{n,a}$ has an elegant reinterpretation.

\begin{lemma}
  \label{lem:eigen}
  The subspace $V_{n,a}$ is $T$-invariant and $T = (n+1) -
  (1-t)\frac{d}{dt}$.  Moreover, the rational functions $(1-t)^i$ where $-n
  \leq i \leq a$ form an eigenbasis for $V_{n,a}$ and the eigenvalue of $T$
  corresponding to $(1-t)^i$ is $n+1+i$.
\end{lemma}

\begin{proof}
  The Binomial Theorem yields $(1-t)^i = (1-t)^{-n} \sum_{j}
  \binom{i+n}{j}(-1)^j t^j$, so the rational functions $(1-t)^i$ for $-n \leq
  i \leq a$ also form triangular basis for $V_{n,a}$.  Hence, we have
  \begin{align*}
    \Bigl( (n+1) - (1-t)\tfrac{d}{dt} \Bigl)\Bigl[ \sum\nolimits_{j} h(j) t^j
    \Bigr] &= \sum\nolimits_{j} (n+1)h(j) t^j - \sum\nolimits_{j} j h(j)
    t^{j-1} + \sum\nolimits_{j} j h(j) t^{j} \\
    &= \sum\nolimits_{j} \bigl( (n+1+j) h(j) - (j+1)h(j+1) \bigr) t^j \\
    &= T \Bigl[ \sum\nolimits_{j} h(j) t^j \Bigr]
  \end{align*}
  and $T \bigl[ (1-t)^i \bigr] = (n+1)(1-t)^i - (1-t)(i)(1-t)^{i-1}(-1) =
  (n+1+i)(1-t)^i$.
\end{proof}

The second lemma in this section calculates the image under $T$ of the Hilbert
function for certain cyclic modules.

\begin{lemma}
  \label{lem:gens}
  Let $\ell \in \NN$ and $i \in \NN$ satisfy $0 \leq \ell \leq n+1$ and $1
  \leq i \leq a+n-\ell+2$.  For the cyclic module $M := S/\langle x_0, x_1,
  \dotsc, x_{\ell-1} \rangle^i$, we have $h_M \in V_{n,a}$, and
  \[
  T \biggl[ \sum\nolimits_{j} h_M(j) t^j \biggr] = T \left[
    (1-t)^{\ell-1-n} \sum_{k = 0}^{i-1} \binom{\ell-1+k}{k} t^k \right] = i
  \binom{\ell-1+i}{i} t^{i-1} (1-t)^{\ell-1-n} \, .
  \]
\end{lemma}

\begin{proof}
  The monomials not in the ideal $\langle x_0, x_1, \dotsc, x_{\ell-1}
  \rangle^i$ form a $\kk$-vector space basis for $M$; see Theorem~15.3 in
  \cite{Eisenbud}.  Since these basis elements in degree $j$ are the disjoint
  union of monomials in $\kk[x_0,x_1, \dotsc, x_{\ell-1}]_{k} \cdot
  \kk[x_{\ell}, x_{\ell+1}, \dotsc, x_{n}]_{j-k}$ where $0 \leq k \leq i-1$,
  we have
  \[
  \sum\nolimits_{j} h_{M}(j) t^j = (1-t)^{\ell-1-n} \sum_{k = 0}^{i-1}
  \binom{\ell-1+k}{k} t^k \, ,
  \]
  so $h_M \in V_{n,a}$ when $i-1+\ell-1 \leq a+n$.  Combining
  Lemma~\ref{lem:eigen} with the Absorption Identity, we obtain
  \begin{align*}
    T \left[ \sum\nolimits_j h_M(j) t^j \right] &= \left( (n+1) -
      (1-t)\frac{d}{dt} \right) \left[ (1-t)^{\ell-1-n} \sum_{k = 0}^{i-1}
      \binom{\ell-1+k}{k} t^k  \right] \\
    &= \ell (1-t)^{\ell-1-n} \sum_{k=0}^{i-1} \binom{\ell-1+k}{k} t^{k} -
    (1-t)^{\ell-n} \sum_{k=1}^{i-1} k \binom{\ell-1+k}{k} t^{k} \\
    &= (1-t)^{\ell-1-n} \left( \ell + \sum_{k=1}^{i-1} (\ell+k)
      \binom{\ell-1+k}{k} t^k - \sum_{k=0}^{i-2} (k+1) \binom{\ell+k}{k+1} t^k
    \right) \\
    &= (1-t)^{\ell-1-n} (\ell-1+i) \binom{\ell+i-2}{i-1} t^{i-1} = i
    \binom{\ell-1+i}{i} t^{i-1} (1-t)^{\ell-1-n} \, . \qedhere
  \end{align*}
\end{proof}

We concluded our trilogy of lemmata with an elementary positivity result.

\begin{lemma}
  \label{lem:pos}
  Any polynomial $f \in \QQ[s]$ of degree $r$ with $r$ distinct negative
  integer roots and a positive leading coefficient is a non-negative
  $\QQ$-linear combination of the polynomials $\binom{s+k}{k}$ for $0 \leq k
  \leq r$.
\end{lemma}

\begin{proof}
  We proceed by induction on $r$.  If $r = 1$, then $f$ is the product of the
  leading coefficient of $f$ and the polynomial $\binom{s+0}{0}$ which
  establishes the base case.  Assume that $r > 1$. Since $f$ has $r$ distinct
  negative integer roots, the smallest root of $f$ equals $-r-\ell$ for some
  $\ell \in \ZZ$ satisfying $\ell \geq 0$.  It follows that $f(s) =
  (s+r+\ell)g(s)$ where $g \in \QQ[s]$ has degree $r-1$, $r-1$ distinct
  negative integer roots, and a positive leading coefficient.  The induction
  hypothesis implies that there exists non-negative $c_0, c_1, \dotsc, c_{r-1}
  \in \QQ$ such that 
  \[
  g(s) = c_0 \binom{s+0}{0} + c_1 \binom{s+1}{1} + \dotsb + c_{r-1}
  \binom{s+r-1}{r-1} \, .
  \]  
  Hence, the Absorption Identity yields
  \begin{align*}
    f(s) = (s+r+\ell) g(s) 
    &= \sum_{k=0}^{r-1} (s+r+\ell) c_k \binom{s+k}{k} \\
    &= \sum_{k=0}^{r-1} (s+k+1) c_k \binom{s+k}{k} + \sum_{k=0}^{r-1}
    (r-1-k+\ell) c_k \binom{s+k}{k} \\
     &= \sum_{k=0}^{r-1} (k+1) c_{k} \binom{s+k+1}{k+1} + \sum_{k=0}^{r-1}
     (r-1-k+\ell) c_k \binom{s+k}{k} \\
    &= \sum_{k=1}^{r} k c_{k-1} \binom{s+k}{k} + \sum_{k=0}^{r-1} (r-1-k+\ell)
    c_k \binom{s+k}{k}
  \end{align*}
  which completes the induction.
\end{proof}

We can now prove Theorem~\ref{thm:two} by showing that $T[Q_{n,a}] = P_{n,a}$.

\begin{proof}[Proof of Theorem~\ref{thm:two}]
  Theorem~\ref{thm:Macaulay} together with Lemma~\ref{lem:eigen} prove that
  $T[Q_{n,a}] \subseteq P_{n,a}$, so it suffices to show that all of the
  extreme rays of $P_{n,a}$ are images under $T$ of elements in $Q_{n,a}$.
  Lemma~\ref{lem:gens} establishes that images under $T$ of Hilbert functions
  for the artinian modules $S/\langle x_0, \dotsc, x_n \rangle^i$ where $1
  \leq i \leq a+1$ are scalar multiples of the polynomials $1, t, \dotsc,
  t^a$.  As in Proposition~\ref{pro:Pcone}, let $\hat{a} := a + \max(1,-a)$,
  fix an appropriate integer partition $\lambda$ or $\mu$, and let $F(t)$
  equal either
  \begin{xalignat*}{3}
    &\sum\limits_{j \geq \hat{a}} \Bigl( p_{\lambda}(j-\hat{a}) \prod_{\ell =
      1}^{\hat{a} - a - 1} (j+\ell) \Bigr) t^j & & \text{or} &\sum\limits_{j
      \geq \hat{a}} \Bigl( p_{\mu}(j-\hat{a}-1) \prod_{\ell = 0}^{\hat{a} - a -
      1} (j+\ell) \Bigr) t^j \, .
  \end{xalignat*}
  We need only exhibit a module $M$ such that the image of its Hilbert series
  under $T$ is a scalar multiple of $F(t)$.

  Since $b := \hat{a} + \lambda_1 + 2r$ is the largest root of
  $p_\lambda(j-\hat{a}-1)$, there is a unique decomposition $F(t) = F_1(t) +
  F_2(t)$ where $F_1(t)$ is a polynomial of degree less than $b$ and $F_2(t)$
  is a power series in which only the terms of degree larger than $b$ have
  nonzero coefficients.  It follows from Lemma~\ref{lem:gens} that the image
  of the Hilbert series an appropriate direct sum $M_1$ of the artinian
  modules $S/\langle x_0, \dotsc, x_n \rangle^i$ for $\hat{a} \leq i \leq b$
  maps to $c_1 \, F_1(t)$ for some positive $c_1 \in \ZZ$.  Thus, if there
  exists a module $M_2$ such that its Hilbert series maps to $c_2 \, F_2(t)$
  for some positive $c_2 \in \ZZ$, then the Hilbert series of the module $M =
  M_1^{(c_2)} \oplus M_2^{(c_1)}$ maps to $c_1c_2 \, F(t)$ under $T$.

  Establishing the existence of $M_2$ reduces by Lemma~\ref{lem:gens} to
  proving that $F_2(t)$ equals a finite non-negative $\QQ$-linear combination
  of the power series $t^{b+1} (1-t)^{-(k+1)} = t^{b+1} \sum_{j \geq 0}
  \binom{j+k}{k} t^{j}$ for $0 \leq k \leq n$.  By construction, we have
  $F_2(t) = t^{b+1} \sum_{j \geq 0} f_2(j) t^j$ where $f_2$ is a polynomial of
  degree $r \leq n$ with $r$ distinct negative integer roots and a positive
  leading coefficient.  Therefore, Lemma~\ref{lem:pos} completes the argument
  by showing that $f_2$ is a non-negative $\QQ$-linear combination of the
  polynomials $\binom{s+k}{k}$ for $0 \leq k \leq r \leq n$.
\end{proof}

\begin{remark}
  The proof of Theorem~\ref{thm:two} is constructive.  However, the procedure
  for creating a module $M$ that generates an extreme ray is rarely effective,
  because the number of cyclic summands used is so large.  Although each
  cyclic summand used has the simple form $S/\langle x_0, x_1, \dotsc,
  x_{\ell-1} \rangle^i$ for some $i \in \NN$ and $1 \leq \ell \leq n+1$, the
  Hilbert function of each individual summand does not belong to $V_{n,a}$.
\end{remark}

\begin{corollary}
  \label{cor:Qcone}
  The closed convex cone $Q_{n,a}$ is the intersection of the closed
  half-spaced defined by the inequalities $(n+j+1) h(j) \geq (j+1) h(j+1)$ for
  $j \in \NN$ and the limiting inequality which asserts that leading
  coefficient of the associated Hilbert polynomial is positive.  If $\hat{a}
  := a + \max(1,-a)$, then the extreme rays of $Q_{n,a}$ are generated by the
  Hilbert functions of the cyclic modules $S/\langle x_0, x_1, \dotsc, x_n
  \rangle^i$ for $1 \leq i \leq a+1$ and the inverse images under $T$ of the
  power series
  \begin{xalignat*}{2}
    &\sum_{j \geq \hat{a}} \Bigl( p_\lambda(j-\hat{a}) \prod_{\ell =
      1}^{\hat{a}-a-1} (j + \ell) \Bigr) t^j \, , & &\sum_{j \geq \hat{a}}
    \Bigl( p_\mu(j-\hat{a}-1) \prod_{\ell=0}^{\hat{a}-a-1} (j + \ell) \Bigr) t^j
  \end{xalignat*}
  where $\lambda$ ranges over all integer partitions with at most $\lfloor
  (n-\hat{a}+a)/2 \rfloor$ parts and $\mu$ ranges over all integer partitions
  with at most $\lfloor (n-\hat{a}+a-1)/2 \rfloor$ parts.
\end{corollary}

\begin{proof}
  This follows immediately from Proposition~\ref{pro:Pcone} and
  Theorem~\ref{thm:two}.
\end{proof}

We end this section with some examples illustrating Corollary~\ref{cor:Qcone}.
When the dimension of the ambient vector space $V_{n,a}$ is small enough, we
can visualize the cone $Q_{n,a}$.

\begin{example}
  If $\dim(V_{n,a}) = 1$, then we have $n = -a$.  The cone $Q_{n,-n}$ is the
  positive $c_n$-axis generated by $(1-t)^{-n}$ which corresponds to the
  $S$-module $S/\langle x_0 \rangle$. \hfill $\diamond$
\end{example}

\begin{example}
  \label{exa:n=0}
  If $n = 0$, then we have $S = \kk[x_0]$ and $a \geq 0$.  Since the
  associated generating functions for elements of $V_{0,a}$ have the form
  $c_{-a} t^a + c_{-a+1} t^{a-1} + \dotsb + c_{-1} t + c_0$, the linear
  half-spaces defining $Q_{0,a}$ are $c_{-j} \geq c_{-j-1}$ for $0 \leq j < a$
  and $c_{-a} \geq 0$.  The extreme rays are generated by $1 + t + \dotsb +
  t^{i-1}$ for $0 \leq i \leq a +1$ which corresponds to the $S$-module
  $S/\langle x_0 \rangle^{i}$.  In particular, $Q_{0,a}$ is a simplicial
  polyhedral cone. \hfill $\diamond$
\end{example}

\begin{example}
  If $\dim(V_{n,a}) = 2$, then we have $a = -n+1$.  The case $n = 0$ is
  described in Example~\ref{exa:n=0}, so we may assume that $n \geq 1$.  Since
  we have
  \[
  \frac{c_{n-1}}{(1-t)^{n-1}} + \frac{c_n}{(1-t)^{n}} = \sum_{j \in \NN}
  \Biggl( c_{n-1} \binom{n+j-2}{n-2} + c_n \binom{n+j-1}{n-1} \Biggr) t^j
  \] 
  the linear half-spaces defining $Q_{n,-n+1}$ are $2(n-1)c_{n-1} +
  (n+j-1)c_{n} \geq 0$ for $j \in \NN$.  In this degenerate case, the two
  linear half-spaces $2c_{n-1} + c_{n} \geq 0$ and $c_n \geq 0$ coming from $j
  = 0$ and $j = \infty$ suffice.  The extreme rays are generated by
  $(1-t)^{-n+1}$ and $-(1-t)^{-n+1} + 2(1-t)^{-n}$ which correspond to the
  $S$-modules $S/\langle x_0,x_1 \rangle$ and $S/ \langle x_0 \rangle^2$.
  Once again, $Q_{n,-n+1}$ is a simplicial polyhedral cone.

  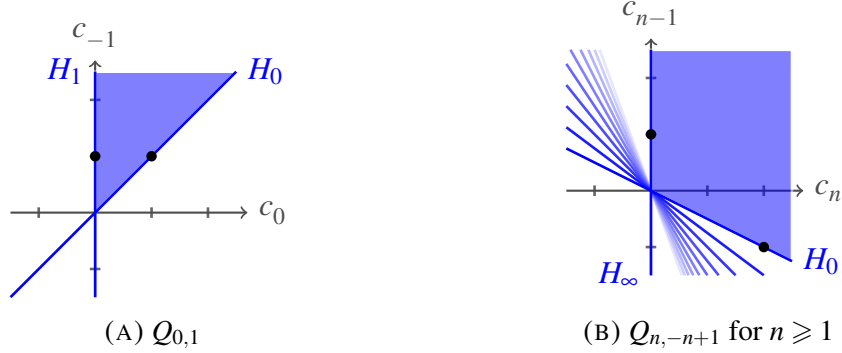
\begin{figure}[ht]
    \centering
    \begin{subfigure}[b]{0.3\textwidth}
      \centering
      \begin{tikzpicture}[x=0.75cm, y=0.75cm, line width=1pt]
        \foreach \x in {-1,1,2} {
          \draw[color=white!40!black] (\x, 2pt) -- (\x, -2pt);
        }
        \foreach \y in {-1,1,2} {
          \draw[color=white!40!black] (2pt,\y) -- (-2pt, \y);
        }
        \draw[color=white!30!black, thick, ->] (-1.5,0) -- (2.7,0) node[right]
        {$c^{}_0$};
        \draw[color=white!30!black, thick, ->] (0,-1.5) -- (0,2.7) node[above]
        {$c^{}_{-1}$};
        \draw[color=white, fill=blue, opacity=0.5] (0,0) -- (2.5,2.5) --
        (0,2.5) -- cycle;
        \draw[color=blue] (-1.5,-1.5) -- (2.5,2.5) node[right] {$H_0$};
        \draw[color=blue] (0,-1.5) -- (0,2.5) node[left] {$H_1$};
        \node[circle,fill=black,inner sep=1.4pt] () at (1,1) {};
        \node[circle,fill=black,inner sep=1.4pt] () at (0,1) {};
      \end{tikzpicture}
      \caption{$Q_{0,1}$}
    \end{subfigure}
    \hspace{6em}
    \begin{subfigure}[b]{0.3\textwidth}
      \centering
      \begin{tikzpicture}[x=0.75cm, y=0.75cm, line width=1pt]
        \foreach \x in {-1,1,2} {
          \draw[color=white!40!black] (\x, 2pt) -- (\x, -2pt);
        }
        \foreach \y in {-1,1,2} {
          \draw[color=white!40!black] (2pt,\y) -- (-2pt, \y);
        }
        \draw[color=white!30!black, thick, ->] (-1.5,0) -- (2.7,0) node[right]
        {$c^{}_{n}$};
        \draw[color=white!30!black, thick, ->] (0,-1.5) -- (0,2.7) node[above]
        {$c^{}_{n-1}$};
        \draw[color=white, fill=blue, opacity=0.5] (0,0) -- (2.5,-1.25) --
        (2.5,2.5) -- (0, 2.5) -- cycle;
        \draw[color=blue!10!white] (.545455,-1.5) -- (-.909091,2.5)
        node[right] {}; 
        \draw[color=blue!20!white] (.6,-1.5) -- (-1,2.5) node[right] {}; 
        \draw[color=blue!30!white] (.666667,-1.5) -- (-1.11111,2.5)
        node[right] {}; 
        \draw[color=blue!40!white] (.75,-1.5) -- (-1.25,2.5) node[right]
        {}; 
        \draw[color=blue!50!white] (.857143,-1.5) -- (-1.42857,2.5)
        node[right] {}; 
        \draw[color=blue!60!white] (1,-1.5) -- (-1.5, 2.25) node[right] {}; 
        \draw[color=blue!70!white] (1.2,-1.5) -- (-1.5, 1.875) node[right]
        {}; 
        \draw[color=blue!80!white] (1.5,-1.5) -- (-1.5, 1.5) node[right]
        {}; 
        \draw[color=blue!90!white] (2,-1.5) -- (-1.5, 1.125) node[right]
        {}; 
        \draw[color=blue] (-1.5,0.75) -- (2.5,-1.25) node[right] {$H_0$};
        \draw[color=blue] (0,2.5) -- (0,-1.5) node[left] {$H_{\infty}$};
        \node[circle,fill=black,inner sep=1.4pt] () at (0,1) {};
        \node[circle,fill=black,inner sep=1.4pt] () at (2,-1) {};
      \end{tikzpicture}
      \caption{$Q_{n,-n+1}$ for $n \geq 1$}
    \end{subfigure}
    \caption{Cones of Hilbert functions when $\dim(V_{n,a}) = 2$.}
    \label{fig:dim2}
  \end{figure}

  In Figure~\ref{fig:dim2}, the supporting hyperplanes $H_j$ are represented
  by blue lines (that fade to white as $j$ increase), the cone is represented
  by the translucent blue region, and the generators of the extreme rays are
  represented by small black circles. \hfill $\diamond$
\end{example}

\begin{example}
  \label{exa:Q_{3,-1}}
  If $n = 3$ and $a = -1$, then $\dim(V_{3,-1}) = 3$.  Since we have
  \begin{align*}
    \frac{c_1}{(1-t)} + \frac{c_2}{(1-t)^2} + \frac{c_3}{(1-t)^3} &= \sum_{j
      \in \NN} \Biggl( c_1 \binom{j}{0} + c_2 \binom{j+1}{1} + c_3
    \binom{j+2}{2} \Biggr) t^j \, ,
  \end{align*}
  the linear half-spaces defining $Q_{3,-1}$ are
  \begin{align*}
    (3+j+1)h(j) -(j+1)h(j+1) = 3 c_1 + 2(j+1) c_2 +
    \tfrac{1}{2}(j+1)(j+2) c_3 &\geq 0  \quad \text{for $j \geq 0$.}
  \end{align*}
  To visualize this closed convex cone, we intersect with the hyperplane $c_1
  + c_2 + c_3 = 1$; for a cyclic module, we have $h(0) = 1$.  Points in this
  cross-section are determined by the coordinates $(c_2, c_1)$, and the linear
  half-spaces in these coordinates are
  \[
  H_j : (j-1)(j+4) c_1 +
  (j-2)(j+1) c_2 - (j+1)(j+2) \leq 0 \quad \text{for $j \geq 0$.}  
  \]  
  As $j \to \infty$, we also obtain $H_{\infty} : c_1 + c_2 -1 \leq 0$.  In
  Figure~\ref{fig:3-1}, the supporting hyperplanes corresponding to $H_j$ are
  represented by blue lines (that fade to white as $j$ increases) and the
  cross-section of the cone is represented by the translucent blue region.

  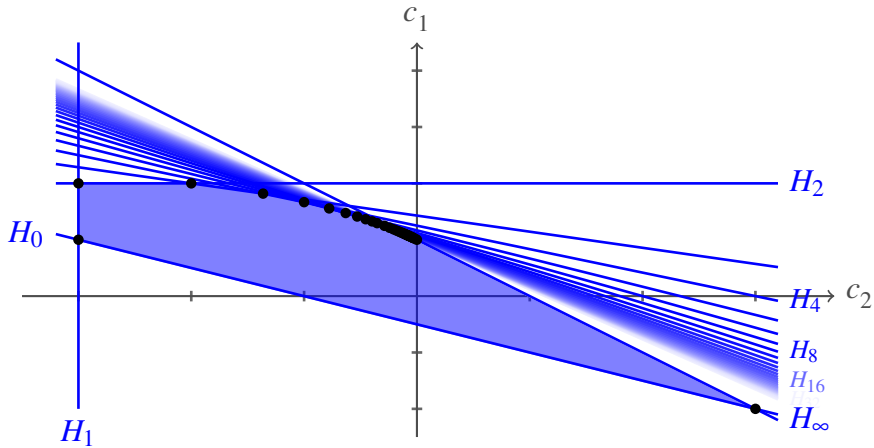
\begin{figure}[ht]
    \begin{tikzpicture}[x=1.5cm, y=0.75cm, line width=1pt]
      \foreach \x in {-3,-2,-1,1,2,3} {
        \draw[color=white!40!black] (\x, 2pt) -- (\x, -2pt);
      }
      \foreach \y in {-2,-1,1,2,3,4} {
        \draw[color=white!40!black] (2pt,\y) -- (-2pt, \y);
      }
      \draw[color=white!30!black, thick, ->] (-3.5,0) -- (3.7,0) node[right]
      {$c^{}_2$};
      \draw[color=white!30!black, thick, ->] (0,-2.5) -- (0,4.5) node[above]
      {$c^{}_1$};
      \draw[color=blue, fill=blue, opacity=0.5] (3,-2) -- (-3,1) -- (-3,2) --
      (-2,2) -- (-1.36364,1.81818) -- (-1,1.66667) -- (-.777778,1.55556) --
      (-.631579,1.47368) -- (-.529412,1.41176) -- (-.454545,1.36364) --
      (-.39759,1.3253) -- (-.352941,1.29412) -- (-.317073,1.26829) --
      (-.287671,1.24658) -- (-.263158,1.22807) -- (-.242424,1.21212) --
      (-.22467,1.19824) -- (-.209302,1.18605) -- (-.195876,1.17526) --
      (-.184049,1.16564) -- (-.173554,1.15702) -- (-.164179,1.14925) --
      (-.155756,1.14221) -- (-.148148,1.1358) -- (-.141243,1.12994) --
      (-.134948,1.12457) -- (-.129187,1.11962) -- (-.123894,1.11504) --
      (-.119015,1.11081) -- (-.114504,1.10687) -- (-.11032,1.1032) --
      (-.10643,1.09978) -- (-.102804,1.09657) -- (-.0994152,1.09357) --
      (-.096242,1.09074) -- (-.0932642,1.08808) -- (-.0904645,1.08557) --
      (-.0878274,1.0832) -- (-.0853392,1.08096) -- (-.0829876,1.07884) --
      (-.0807617,1.07682) -- (-.0786517,1.07491) -- (-.0766488,1.07308) --
      (-.0747452,1.07135) -- (-.0729335,1.06969) -- (-.0712074,1.06811) --
      (-.0695609,1.0666) -- (-.0679887,1.06516) -- (-.0664858,1.06377) --
      (-.0650477,1.06245) -- (-.0636704,1.06117) -- (-.0623501,1.05995) --
      (-.0610834,1.05878) -- (-.059867,1.05765) -- (-.058698,1.05656) --
      (-.0575737,1.05552) -- (-.0564916,1.05451) -- (-.0554493,1.05354) --
      (-.0544448,1.0526) -- (-.0534759,1.05169) -- (-.0525409,1.05082) --
      (-.051638,1.04997) -- (-.0507655,1.04915) -- (-.049922,1.04836) --
      (-.049106,1.0476) -- (-.0483163,1.04685) -- (-.0475515,1.04613) --
      (-.0468105,1.04543) -- (-.0460922,1.04476) -- (-.0453956,1.0441) --
      (-.0447197,1.04346) -- (-.0440636,1.04284) -- (-.0434265,1.04224) --
      (-.0428076,1.04165) -- (-.042206,1.04108) -- (-.041621,1.04053) --
      (-.0410521,1.03999) (0,-1) -- cycle;
      \draw[color=white,fill=white] (0,1) -- (3.2,-1.83333) -- (3.2,-2.2) --
      cycle; \draw[color=blue!05!white] (-3.2, 3.84409) -- (3.2, -1.83333)
      node[right] {\tiny $H_{32}$}; 
      \draw[color=blue!06!white] (-3.2, 3.8339) -- (3.2, -1.82248) node[right]
      {}; 
      \draw[color=blue!07!white] (-3.2, 3.82312) -- (3.2, -1.81095)
      node[right] {}; 
      \draw[color=blue!08!white] (-3.2, 3.81169) -- (3.2, -1.7987) node[right]
      {}; 
      \draw[color=blue!09!white] (-3.2, 3.79954) -- (3.2, -1.78565)
      node[right] {}; 
      \draw[color=blue!10!white] (-3.2, 3.7866) -- (3.2, -1.77171) node[right]
      {}; 
      \draw[color=blue!12!white] (-3.2, 3.7728) -- (3.2, -1.7568) node[right]
      {}; 
      \draw[color=blue!15!white] (-3.2, 3.75805) -- (3.2, -1.7408) node[right]
      {}; 
      \draw[color=blue!20!white] (-3.2, 3.74224) -- (3.2, -1.7236) node[right]
      {}; 
      \draw[color=blue!25!white] (-3.2, 3.72525) -- (3.2, -1.70505)
      node[right] {}; 
      \draw[color=blue!30!white] (-3.2, 3.70696) -- (3.2, -1.68498)
      node[right] {}; 
      \draw[color=blue!35!white] (-3.2, 3.6872) -- (3.2, -1.6632) node[right]
      {}; 
      \draw[color=blue!40!white] (-3.2, 3.66579) -- (3.2, -1.63947)
      node[right] {}; 
      \draw[color=blue!45!white] (-3.2, 3.64251) -- (3.2, -1.61353)
      node[right] {}; 
      \draw[color=blue!50!white] (-3.2, 3.61711) -- (3.2, -1.58503)
      node[right] {}; 
      \draw[color=blue!55!white] (-3.2, 3.58929) -- (3.2, -1.55357)
      node[right] {}; 
      \draw[color=blue!60!white] (-3.2, 3.55867) -- (3.2, -1.51867)
      node[right] {\scriptsize $H_{16}$}; 
      \draw[color=blue!65!white] (-3.2, 3.52481) -- (3.2, -1.4797) node[right]
      {}; 
      \draw[color=blue!70!white] (-3.2, 3.48718) -- (3.2, -1.4359) node[right]
      {}; 
      \draw[color=blue!75!white] (-3.2, 3.4451) -- (3.2, -1.38627) node[right]
      {}; 
      \draw[color=blue!80!white] (-3.2, 3.39773) -- (3.2, -1.32955)
      node[right] {}; 
      \draw[color=blue!85!white] (-3.2, 3.344) -- (3.2, -1.264) node[right]
      {};
      \draw[color=blue!90!white] (-3.2, 3.28254) -- (3.2, -1.1873) node[right]
      {}; 
      \draw[color=blue!95!white] (-3.2, 3.21154) -- (3.2, -1.09615) node[right]
      {}; 
      \draw[color=blue] (-3.2, 3.12857) -- (3.2, -.985714) node[right]
      {\footnotesize $H_8$}; 
      \draw[color=blue] (-3.2, 3.0303) -- (3.2, -.848485) node[right] {}; 
      \draw[color=blue] (-3.2, 2.912) -- (3.2, -.672) node[right] {}; 
      \draw[color=blue] (-3.2, 2.76667) -- (3.2, -.433333) node[right] {}; 
      \draw[color=blue] (-3.2, 2.58333) -- (3.2, -.0833333) node[right] 
      {\small $H_4$}; 
      \draw[color=blue] (-3.2, 2.34286) -- (3.2, .514286) node[right] {}; 
      \draw[color=blue] (-3.2, 2) -- (3.2, 2) node[right] {$H_2$}; 
      \draw[color=blue] (-3,4.5) -- (-3,-2.0) node[below] {$H_1$}; 
      \draw[color=blue] (3.2, -2.1) -- (-3.2, 1.1) node[left] {$H_0$}; 
      \draw[color=blue] (-3.2,4.2) -- (3.2,-2.2) node[right] {$H_{\infty}$};
      \node[circle,fill=black,inner sep=1.4pt] () at (3,-2) {};
      \node[circle,fill=black,inner sep=1.4pt] () at (0,1) {};
      \node[circle,fill=black,inner sep=1.4pt] () at (-3,1) {};
      \node[circle,fill=black,inner sep=1.4pt] () at (-3,2) {};
      \node[circle,fill=black,inner sep=1.4pt] () at (-2,2) {};
      \node[circle,fill=black,inner sep=1.4pt] () at (-1.36364,1.81818) {};
      \node[circle,fill=black,inner sep=1.4pt] () at (-1,1.66667) {};
      \node[circle,fill=black,inner sep=1.4pt] () at (-.777778,1.55556) {};
      \node[circle,fill=black,inner sep=1.4pt] () at (-.631579,1.47368) {};
      \node[circle,fill=black,inner sep=1.4pt] () at (-.529412,1.41176) {};
      \node[circle,fill=black,inner sep=1.4pt] () at (-.454545,1.36364) {};
      \node[circle,fill=black,inner sep=1.4pt] () at (-.39759,1.3253) {};
      \node[circle,fill=black,inner sep=1.4pt] () at (-.352941,1.29412) {};
      \node[circle,fill=black,inner sep=1.4pt] () at (-.287671,1.24658) {};
      \node[circle,fill=black,inner sep=1.4pt] () at (-.242424,1.21212) {};
      \node[circle,fill=black,inner sep=1.4pt] () at (-.209302,1.18605) {};
      \node[circle,fill=black,inner sep=1.4pt] () at (-.184049,1.16564) {};
      \node[circle,fill=black,inner sep=1.4pt] () at (-.164179,1.14925) {};
      \node[circle,fill=black,inner sep=1.4pt] () at (-.148148,1.1358) {};
      \node[circle,fill=black,inner sep=1.4pt] () at (-.134948,1.12457) {};
      \node[circle,fill=black,inner sep=1.4pt] () at (-.123894,1.11504) {};
      \node[circle,fill=black,inner sep=1.4pt] () at (-.114504,1.10687) {};
      \node[circle,fill=black,inner sep=1.4pt] () at (-.10643,1.09978) {};
      \node[circle,fill=black,inner sep=1.4pt] () at (-.0994152,1.09357)
      {};
      \node[circle,fill=black,inner sep=1.4pt] () at (-.0932642,1.08808)
      {};
      \node[circle,fill=black,inner sep=1.4pt] () at (-.0878274,1.0832) {};
      \node[circle,fill=black,inner sep=1.4pt] () at (-.0829876,1.07884)
      {};
      \node[circle,fill=black,inner sep=1.4pt] () at (-.0786517,1.07491)
      {};
      \node[circle,fill=black,inner sep=1.4pt] () at (-.0747452,1.07135)
      {};
      \node[circle,fill=black,inner sep=1.4pt] () at (-.0712074,1.06811)
      {};
      \node[circle,fill=black,inner sep=1.4pt] () at (-.0679887,1.06516)
      {};
      \node[circle,fill=black,inner sep=1.4pt] () at (-.0650477,1.06245)
      {};
      \node[circle,fill=black,inner sep=1.4pt] () at (-.0623501,1.05995)
      {};
      \node[circle,fill=black,inner sep=1.4pt] () at (-.051638,1.04997) {};
      \node[circle,fill=black,inner sep=1.4pt] () at (-.0440636,1.04284)
      {};
      \node[circle,fill=black,inner sep=1.4pt] () at (-.0384255,1.03749)
      {};
      \node[circle,fill=black,inner sep=1.4pt] () at (-.0340657,1.03333)
      {};
      \node[circle,fill=black,inner sep=1.4pt] () at (-.0305939,1.02999)
      {};
      \node[circle,fill=black,inner sep=1.4pt] () at (-.0151492,1.015) {};
    \end{tikzpicture}
    \caption{Cyclic cross-section of $Q_{3,-1}$}
    \label{fig:3-1}
  \end{figure}

  The extreme points of the cross-section are represented by small black
  circles in Figure~\ref{fig:3-1}.  More precisely, the supporting hyperplanes
  corresponding to $H_i$ and $H_{i+1}$ meet at the point $(c_2,c_1) =
  \tfrac{3}{i^2+2} \left( -(i+2), \frac{1}{3} (i+1)(i+2) \right)$ for $i \geq
  0$, and the supporting hyperplanes corresponding to $H_0$ and $H_{\infty}$
  meet at the point $(c_2,c_1) = (3,-2)$.  As $i \to \infty$, we also obtain
  the point $(0,1)$.  Hence, the extreme rays of $Q_{3,-1}$ are generated by
  \begin{xalignat*}{4}
    & \frac{1}{(1-t)} \, , &&-\frac{2}{(1-t)} + \frac{3}{(1-t)^2} \, , &&
    \text{and} && \frac{3}{i^2+2} \left( \frac{(i+1)(i+2)}{3(1-t)} -
      \frac{(i+2)}{(1-t)^2} + \frac{2}{(1-t)^3} \right) \, .
  \end{xalignat*}
  Moreover, these extreme rays correspond to integer partitions with at most
  $1$-part:
  \begin{xalignat*}{3}
    T \left[ \frac{1}{(1-t)} \right] &= \sum_{j \in \NN} t^j &
    &\longleftrightarrow &\lambda &= \varnothing \\
    T \left[ -\frac{2}{(1-t)} + \frac{3}{(1-t)^2} \right] &=
    \sum_{j \in \NN} j t^j & &\longleftrightarrow & \lambda &= \varnothing \\
    T \left[ \frac{(i+1)(i+2)}{3(1-t)} - \frac{(i+2)}{(1-t)^2} +
      \frac{2}{(1-t)^3} \right] &= \sum_{j \in \NN} (j-i)(j-i-1) t^j &
    &\longleftrightarrow &\lambda &= (i)  \, .
  \end{xalignat*}
  The cone $Q_{3,-1}$ is neither simplicial nor polyhedral.

  The minimal number of generators for the modules lying on the extreme rays
  is unbounded.  Specifically, by considering the linear term, we see that the
  smallest multiple of the rational function 
  \[
  \frac{3}{i^2+2} \left( \frac{(i+1)(i+2)}{3(1-t)} - \frac{(i+2)}{(1-t)^2} +
    \frac{2}{(1-t)^3} \right)
  \] 
  that could be the Hilbert function of a module has constant term $i^2+2$.
  Hence, any module that corresponds to a point on this ray has at least
  $i^2+2$ generators in degree $0$. \hfill $\diamond$
\end{example}

\section{Modules with bounded regularity}
\label{sec:regularity}

\noindent
This final section examines our third cone of Hilbert functions.  By bounding
the Castelnuovo--Mumford regularity of $S$-modules, we provide an alternative
condition which guarantees that the Hilbert functions lie in a
finite-dimensional vector space.  To enumerate the supporting hyperplanes and
extreme rays for the cone of Hilbert functions with bounded regularity, we use
the natural projection from the cone of Betti tables.

For a finitely generated $\NN$-graded $S$-module $M$, the graded Betti numbers
are defined by $\beta_{i,j}(M) := \dim_{\kk} \bigl( \Tor_i(M,\kk)_j \bigr)$,
and we have $\beta_{i,j}(M) = 0$ for all $i > n+1$; see Theorem~1.1 in
\cite{EisenbudSyz}.  The graded Betti numbers of $M$ determine its Hilbert
series via the formula
\begin{equation}
  \label{eq:bettiHS}
  \sum_{j \in \NN} h_{M}(j) t^j = \frac{\sum_{j \in \NN}
    \sum_{i=0}^{n+1} \beta_{i,j}(M) t^j}{(1-t)^{n+1}} \, .
\end{equation}
The Betti table $\beta(M)$ is the matrix in $\bigoplus_{j = - \infty}^{\infty}
\bigoplus_{i=0}^{n+1} \QQ$ whose entry in the $j$-th row and $i$-th column is
$\beta_{i,i+j}(M)$; see Proposition~1.9 in \cite{EisenbudSyz} for an
explanation of this convention.  The Castelnuovo--Mumford regularity is the
largest index of a nonzero row in the Betti table $\beta(M)$ or equivalently
$\reg(M) := \max \{ j \in \ZZ : \beta_{i,i+j}(M) \neq 0 \}$.  The Hilbert
function $h_M(j)$ equals the Hilbert polynomial $q_M(j)$ for all $j >
\reg(M)$; see Theorem~4.2 in \cite{EisenbudSyz}.  Hence, if $m \geq \reg(M)$,
then Equation~\eqref{eq:bettiHS} shows that $h_M \in V_{n+1,m}$.

\begin{definition}
  \label{def:regularity}
  Let $R_{n,m}$ denote the closed convex hull in $V_{n+1,m}$ of the Hilbert
  functions of finitely generated $\NN$-graded $S$-modules that are generated
  in degree zero, have no free summands, and have Castelnuovo--Mumford
  regularity at most $m$.  Informally, we say that $R_{n,m} \subset V_{n+1,m}$
  is the \define{cone of Hilbert functions with bounded regularity}.
\end{definition}

As in Section~\ref{sec:a-invariant}, let $q_h \in \QQ[s]$ be the Hilbert
polynomial of the sequence $h \in R_{n,m}$.  The backward difference operator
$\nabla \colon \QQ[s] \to \QQ[s]$ is defined by $\nabla q (s) := q(s) -
q(s-1)$ where $q \in \QQ[s]$.  We write $\nabla^i$ for the $i$-fold
composition of $\nabla$ with itself.

\begin{proof}[Proof of Theorem~\ref{thm:three}]
  We first show that the cone $R_{n,m}$ is generated by the Hilbert functions
  of the cyclic modules appearing in the list \eqref{eq:cyclic}.  Our indirect
  proof exploits the Betti tables for certain modules over the smaller
  polynomial ring $S' := S/\langle x_n \rangle = \kk[x_0,x_1, \dotsc,
  x_{n-1}]$.

  Let $\Psi$ be the linear map from the rational vector space of Betti tables
  for $S$-modules to the rational vector space of Betti tables for
  $S'$-modules defined by $\bigl( \Psi(\beta) \bigr)_{i,j} := j
  \beta_{i+1,j}$; compare with Definition~4.5 in \cite{Soderberg}.  Following
  Definition~2.1 in \cite{BS}, the pure Betti table with degree sequence $d_0
  < d_1 < \dotsb < d_e$ satisfies $\beta_{i,d_i} = \prod\nolimits_{j \neq i}
  \frac{1}{|d_j-d_i|}$ for $0 \leq i \leq e$.  Hence, the map $\Psi$ sends the
  pure Betti table with degree sequence $0 < d_1 < d_2 < \dotsb < d_e$ to the
  pure Betti table with degree sequence $d_1 < d_2 < \dotsb < d_e$.  Since
  Theorem~3.7 and Theorem~4.1 in \cite{BS} establish that the closed convex
  cones of Betti tables are generated by the pure Betti tables, the map $\Psi$
  induces a surjection from the closed convex cone of Betti tables for
  $S$-modules to the closed convex cone of Betti tables for
  $S'$-modules. Moreover, the kernel of $\Psi$ is generated by the Betti table
  for the free module $S$.  Therefore, the Betti tables for the modules
  associated to the generators of the cone $R_{n,m}$ correspond to the Betti
  tables for finitely generated $\NN$-graded $S'$-modules that are generated
  in degree at least $0$ and have regularity at most $m$.

  Consider a finitely generated $\NN$-graded $S'$-module $M'$ that is
  generated in degrees at least $0$ and has regularity at most $m$.  Any such
  module $M'$ has the same Hilbert function as the $S'$-module
  \begin{equation}
    \label{eq:M''}
    M'' := \bigoplus_{j=0}^{m-1} \kk(-j)^{\oplus h_{M'}(j)} \oplus M'_{\geq m}
  \end{equation}
  where the truncation $M'_{\geq m}$ equals $\bigoplus_{j \geq m} M'_{j}$.
  Proposition~1.1 and Theorem~1.2 in \cite{EG} establish that $M'_{\geq m}$
  has a linear resolution such that $\beta_{i,i+m}(M'_{\geq m}) =
  \beta_{i,i+m}(M'')$ for $0 \leq i \leq n$.  The Koszul complex is also
  linear, in addition to being the minimal free resolution of the $S'$-module
  $\kk$, so it follows that $\beta_{i,i+j}(M'') = h_{M'}(j) \binom{n}{i}$ for
  $0 \leq j < m$.  Since the Betti table of a direct sum is the sum of Betti
  tables and the relation \eqref{eq:bettiHS} holds, we deduce that the Hilbert
  function of $M'$ can be expressed as a non-negative integer combination of
  Hilbert functions of modules with linear resolutions.

  As each cyclic modules appearing in the list \eqref{eq:cyclic} is the
  quotient of $S$ by a Borel-fixed ideal, the minimal free resolution is given
  by an appropriate Eliahou-Kervaire resolution; see \S2.3 in \cite{MS}.  In
  particular, Theorem~2.18 in \cite{MS} implies that the map $\Psi$ sends
  $\beta(S/\langle x_0, x_1, \dotsc, x_\ell \rangle^d)$ to a pure resolution
  with degree sequence $d < d+1 < \dotsb < d+\ell$.  In other words, the image
  of $\beta(S/\langle x_0, x_1, \dotsc, x_\ell \rangle^d)$ is a Betti table of
  a linear resolution.  Taking the inverse image under $\Psi$ for our
  expression for the Hilbert function of $M'$, we conclude that each generator
  of the cone $R_{n,m}$ is a non-negative rational combination of the Hilbert
  functions of the cyclic modules appearing in the list \eqref{eq:cyclic}.

  We next describe the supporting hyperplanes to the cone $R_{n,m}$.  As in
  Proposition~\ref{pro:Pcone}, the Binomial Theorem establishes both $t^k =
  (1-t)^{-n-1} \sum_i \binom{n+1}{i} (-1)^i t^{k+i}$ for $0 \leq k \leq m$ and
  $(1-t)^{-\ell}t^{m+1} = (1-t)^{-n-1} \sum_i \binom{n-\ell+1}{i} (-1)^i
  t^{i+m+1}$ for $1 \leq \ell \leq n+1$.  Hence, the rational functions $1,t,
  \dotsc, t^{m}, (1-t)^{-1}t^{m+1}, (1-t)^{-2} t^{m+1}, \dotsc, (1-t)^{-n-1}
  t^{m+1}$ form a triangular basis for $V_{n+1,m}$.  Let $(c_0, c_1, \dotsc,
  c_{m}, c_{-1}, c_{-2}, \dotsc, c_{-n-1})$ denote the coordinates of $h \in
  V_{n+1,m}$ with respect to this ordered basis.  Lemma~\ref{lem:gens} implies
  that the Hilbert series of the $S$-module $M_{n,i} := S/\langle x_0, x_1,
  \dotsc, x_n \rangle^{i}$ for $1 \leq i \leq m+1$ is
  \begin{align*}
    \sum_{j} h_{M_{n,i}}(j) t^j &= \sum_{k=0}^{i-1} \binom{n+k}{k} t^k =
    \sum_{k=0}^{i-1} \binom{n+k}{n} t^k \, ,
  \end{align*}
  so the coordinates are $c_k = \binom{n+k}{k}$ for $0 \leq k \leq i-1$ and
  $c_k = 0$ for $i \leq k \leq m$ or $k < 0$. Similarly, the Hilbert series of
  $M_{\ell,m+1} := S/\langle x_0, x_1, \dotsc, x_{n+1-\ell} \rangle^{m+1}$ for
  $1 \leq \ell \leq n+1$ is
  \begin{align*}
    \sum_{j} h_{M_{\ell,m+1}}(j) t^j &= (1-t)^{1-\ell} \sum_{k=0}^{m}
    \binom{n+1-\ell+k}{k} t^k \\ &= \left( (1-t)^{-\ell} \sum_{k=0}^{m}
      \binom{n-\ell+k}{k} t^k \right) - \binom{n+1-\ell+m}{m}
    (1-t)^{-\ell} t^{m+1} \\
    &= \sum_{j} h_{M_{\ell+1,m+1}}(j) t^j - \binom{n+1-\ell+m}{m}
    (1-t)^{-\ell} t^{m+1} \, ,
  \end{align*}
  so the coordinates are $c_k = \binom{n+k}{n}$ for $0 \leq k \leq m$, $c_{-k}
  = \binom{n+1-k+m}{m}$ for $1 \leq k \leq \ell$, and $c_{-k} = 0$ for $\ell
  \leq k \leq n+1$.  Since the coordinate vectors are all truncations of the
  coordinate vector of $h_{M_{n+1,m+1}} \in R_{n,m}$, the inequalities
  defining this cone are simply 
  \begin{xalignat*}{2}
    \frac{c_k}{\binom{n+k}{n}} &\geq
    \frac{c_{k+1}}{\binom{n+k+1}{n}} \quad \text{for $0 \leq k \leq
      m-1$,} &
    \frac{c_m}{\binom{n+m}{n}} &\geq
    \frac{c_{-1}}{\binom{n+m}{m}} \, ,\\
    \frac{c_{-k}}{\binom{n+1-k+m}{m}} &\geq
    \frac{c_{-k-1}}{\binom{n-k+m}{m}} \quad \text{for $1 \leq k \leq
      n$, and} &
    c_{-n-1} &= 0 \, .
  \end{xalignat*}
  The equation $c_{-n-1} = 0$ implies that $R_{n,m} \subset V_{n,m}$.
 
  To complete the proof, we explicitly relate the coordinates to the Hilbert
  function.  For $h \in V_{n+1,m}$, we have
  \[
  \sum_{j} h(j) t^j = c_0 + c_1 t + \dotsb + c_m t^m + c_{-1}
  \frac{t^{m+1}}{(1-t)} + c_{-2} \frac{t^{m+1}}{(1-t)^2} + \dotsb + c_{-n-1}
  \frac{t^{m+1}}{(1-t)^{n+1}} \, ,
  \]
  so $h(j) = c_j$ for $0 \leq j \leq m$ and the Generalized Binomial Theorem
  shows that
  \[
  q_h(s) = \sum\limits_{k=1}^{n+1} c_{-k} \binom{k+s-m-2}{k-1} \in
  \QQ[s] \, .
  \]  
  The Addition Formula for binomial coefficients yields $\nabla^i q_h (s) =
  \sum\limits_{k=i+1}^{n+1} c_{-k} \binom{k+s-m-2-i}{k-1-i}$ from which we
  obtain $\nabla^i q_h (m) = c_{-i-1}$ for $0 \leq i \leq n$.  Using the
  Absorption Identity, the inequalities defining the cone $R_{n,m} \subset
  V_{n,m}$ become
  \begin{align*}
    (n + j +1) h(j) & \geq (j+1) h(j+1) && \text{for $0 \leq j \leq m-1$}
    \\
    h(m) & \geq q_h(m) && \text{and,} \\
    (n+1-i) \nabla^i q_h (m) &\geq (n+m+1-i) \nabla^{i+1} q_h (m) &&
    \text{for $0 \leq i \leq n-1$.}  \qedhere
  \end{align*}
\end{proof}  

\begin{remark}
  The coefficients appearing the supporting hyperplanes of $R_{n,m}$ have an
  intrinsic interpretation in terms of the Hilbert function of the underlying
  ring.  Specifically, the Hilbert polynomial of $S$ is $q_S(s) =
  \binom{n+s}{n}$ and the Addition Formula yields $\nabla^i q_S(m) =
  \binom{n+m-i}{m}$, so $R_{n,m}$ is the intersection of the closed
  half-spaces given by the inequalities:
  \begin{xalignat*}{3}
    \frac{h(j)}{h_S(j)} &\geq \frac{h(j+1)}{h_S(j)} && \text{for $0 \leq j < m$,} &
    \frac{h(m)}{h_S(m)} &\geq \frac{q_h(m)}{q_S(m)} \, ,  \\
    \frac{\nabla^i q_h (m)}{\nabla^i q_S(m)} &\geq \frac{\nabla^{i+1} q_h
      (m)}{\nabla^{i+1} q_S(m)} && \text{for $0 \leq i < n$, and} & \nabla^n q_h(m)
    &=  0  \, .
  \end{xalignat*}
\end{remark}

\begin{remark}
  Corollary~\ref{cor:Qcone} and Theorem~\ref{thm:three}, together with
  Lemma~\ref{lem:eigen}, establish that $R_{n,m} \subseteq Q_{n,m}$.
  Moreover, the artinian cyclic modules $S/\langle x_0, x_1, \dotsc, x_n
  \rangle^i$ for $1 \leq i \leq m+1$ generate extreme rays in both cones.
  However, the simplicial cone $R_{n,m}$ is generally a proper subcone of
  $Q_{n,m}$.
\end{remark}

The techniques used in the proof of Theorem~\ref{thm:three} lead to
descriptions of other cones closed related to $R_{n,m}$.

\begin{remark}
  Restricting to modules of dimension at most $d$ and regularity at most $m$
  yields a subcone of $R_{n,m}$ generated by the Hilbert functions of the
  cyclic modules:
  \begin{equation*}
    \begin{split}
      \frac{S}{\langle x_0, x_1, \dotsc, x_{n} \rangle}, &\frac{S}{\langle x_0,
        x_1, \dotsc, x_{n} \rangle^2}, \dotsc, \frac{S}{\langle x_0, x_1,
        \dotsc, x_{n} \rangle^{m}}, \\
      &\frac{S}{\langle x_0, x_1, \dotsc, x_{n} \rangle^{m+1}},
      \frac{S}{\langle x_0, x_1, \dotsc, x_{n-1} \rangle^{m+1}}, \dotsc,
      \frac{S}{\langle x_0, x_1, \dotsc, x_{n-d} \rangle^{m+1}} \, .
    \end{split}
  \end{equation*}
  For the dual description, we need to add the equalities 
  \[
  \nabla^{d} q_h(m) = \nabla^{d+1} q_h(m) = \dotsb = \nabla^{n} q_h(m) = 0 \,
  .
  \]  
  Similarly, one can describe the restriction to modules with projective
  dimension at most $\ell$ and regularity at most $m$ by relating it to
  $R_{\ell-1,m}$ via the backward difference operator $\nabla^{n+1-\ell}$.
\end{remark}

The techniques also yield explicit bounds for the Betti numbers of modules
with a fixed Hilbert function and bounded regularity.

\begin{proposition}
  \label{pro:BettiBounds}
  If the $S$-module $M$ is generated in degree zero, has no free summands, and
  has Castelnuovo--Mumford regularity at most $m$, then the Betti numbers are
  bounded by the inequalities
  \begin{align*}
    \beta_{i,i+j}(M) &\leq \frac{1}{i+j} \binom{n}{i-1} \Bigl( (n+1+j) h_M(j)
    - (j+1) h_M(j+1) \Bigr) 
    = \tfrac{1}{i+j} \tbinom{n}{i-1} \bigl( T[h_M] \bigr)(j) \\
    \intertext{for $0 \leq j < m$, $1 \leq i \leq n+1$, and} \beta_{i,i+m}(M)
    &\leq \frac{n+m+1}{i+m} \binom{n}{i-1} h_M(m) + \sum\limits_{k=1}^{i}
    (-1)^k \binom{n+1}{i-k} h_M(m+k)
  \end{align*}
  for $1 \leq i \leq n+1$.  Moreover, these bounds are sharp for some positive
  multiple of the Hilbert function $h_M$.
\end{proposition}

\begin{proof}
  Theorem~1 in \cite{Hulett} establishes that the module $M''$, defined in
  Equation~\eqref{eq:M''}, has the largest possible Betti numbers among all
  $S'$-modules with a given Hilbert function (up to scaling) and regularity at
  most $m$.  The matrix with respect to the standard basis of the linear map
  $\Psi$ has non-negative entries, so the Betti table $\Psi^{-1}\bigl(
  \beta(M'') \bigr)$ is maximal among all $S$-modules that are generated in
  degree zero, have no free summands, have regularity at most $m$, and have a
  given Hilbert function.  We compute $\Psi^{-1}\bigl( \beta(M'') \bigr)$ from
  the expansion of $h_M$ as a non-negative linear combination of the extreme
  rays.  Ordering the extreme rays as in the list \eqref{eq:cyclic}, the
  coefficients $\alpha_0, \alpha_1, \dotsc, \alpha_m, \alpha_{-1},
  \alpha_{-2}, \dotsc, \alpha_{-n}$ in the unique such expansion are
  \begin{xalignat*}{4}
    \alpha_j &= \frac{h_M(j)}{\binom{n+j}{n}} -
    \frac{h_M(j+1)}{\binom{n+j+1}{n}} && \text{for $0 \leq j < m$,} &
    \alpha_m &= \frac{h_M(m) - q_M(m)}{\binom{n+m}{n}} && \text{and,} \\
    \alpha_{-i} &= \frac{\nabla^{i-1} q_M(m)}{\binom{n+m+1-i}{m}} -
    \frac{\nabla^{i} q_M(m)}{\binom{n+m-i}{m}} && \text{for $1 \leq i \leq n$.}
  \end{xalignat*}
  Since each of the extreme rays corresponds to a cyclic $S$-module with a
  linear resolution and well-known Betti numbers (see Theorem~4.1.15 in
  \cite{BH}), the Absorption Identity gives
  \begin{align*}
    \beta_{i,i+j}(M) &\leq \alpha_j \cdot \beta_{i,i+j}\left( \frac{S}{\langle
        x_0, x_1, \dotsc, x_n \rangle^{j+1}} \right) \\
    &= \left( \frac{h_M(j)}{\binom{n+j}{n}} -
      \frac{h_M(j+1)}{\binom{n+j+1}{n}} \right) \cdot \left( \frac{i}{i+j}
    \right) \binom{n+j+1}{n+1} \binom{n+1}{i} \\
    &= \Bigl( (n+j+1) h_M(j) - (j+1) h_M(j+1) \Bigr) \left( \frac{1}{i+j}
    \right) \binom{n}{n-i} \\
    &= \frac{1}{i+j} \binom{n}{n-i} \bigl( T[h_M] \bigr) (j)
  \end{align*} 
  for $0 \leq j < m$, $1 \leq i \leq n+1$.  Since we have $\nabla^n q_M(m) =
  0$, the Absorption Identity and the Addition Formula give
  \begin{align*}
    \beta_{i,i+m}(M) &\leq \alpha_{m} \cdot \beta_{i,i+m}\left(
      \frac{S}{\langle x_0, x_1, \dotsc, x_n \rangle^{m+1}} \right) +
    \sum_{k=1}^{n} \alpha_{-k} \cdot \beta_{i,i+m}\left( \frac{S}{\langle
        x_0, x_1, \dotsc, x_{n-k} \rangle^{m+1}} \right) \\
    &= \left( \frac{h_M(m) - q_M(m)}{\binom{n+m}{n}} \right) \cdot \left(
      \frac{i}{i+m} \right) \binom{n+m+1}{n+1} \binom{n+1}{i} \\
    &\relphantom{wii} + \sum_{k=1}^{n} \left( \frac{\nabla^{k-1}
        q_M(m)}{\binom{n+m+1-k}{m}} - \frac{\nabla^{k}
        q_M(m)}{\binom{n+m-k}{m}} \right) \cdot \left( \frac{i}{i+m} \right)
    \binom{n+1-k +m}{n+1-k} \binom{n+1-k}{i} \\
    &= \frac{(n+m+1) h_M(m)}{i+m} \binom{n}{i-1} - \frac{(n+m+1) q_M(m)}{i+m}
    \binom{n}{i-1} \\
    &\relphantom{wii} + \sum_{k=1}^{n} \frac{(n+1-k) \nabla^{k-1} q_M(m)}{i+m}
    \binom{n-k}{i-1} - \sum_{k=1}^{n} \frac{(n+m+1-k) \nabla^{k} q_M(m) }{i+m}
    \binom{n-k}{i-1} \\
    &= \frac{n+m+1}{i+m} \binom{n}{i-1}  h_M(m) \\
    &\relphantom{wii} + \sum_{k=0}^{n} \frac{\nabla^{k} q_M(m) }{i+m} \Biggl(
      (n-k) \binom{n-k-1}{i-1} - (n+m+1-k) \binom{n-k}{i-1} \Biggr) \\
    &= \frac{n+m+1}{i+m} \binom{n}{i-1} h_M(m) - \sum_{k=0}^{n} \nabla^{k}
    q_M(m) \binom{n-k}{i-1}
  \end{align*}
  for $1 \leq i \leq n+1$.  Combining the binomial identity
  $\displaystyle\sum\limits_{k=0}^r \binom{r-k}{\ell} \binom{k}{i} =
  \binom{r+1}{\ell+i+1}$ with the higher-order difference formula yields
  \begin{align*}
    \sum_{k=0}^{n} \nabla^{k} q_M(m) \binom{n-k}{i-1} &= \sum_{k=0}^{n}
    \sum_{\ell=0}^{k} (-1)^\ell \binom{k}{\ell} \binom{n-k}{i-1} q_M(m-\ell)
    \\
    &= \sum_{\ell=0}^{n} (-1)^\ell \binom{n+1}{\ell+i} q_M(m-\ell) \, .
  \end{align*}
  Since $\nabla^{n+1}q_M(s) = 0$ and $q_M(j) = h_M(j)$ for all $j > m$, we
  obtain
  \begin{align*}
    \sum_{k=0}^{n} \nabla^{k} q_M(m) \binom{n-k}{i-1} &=
    \sum_{\ell=-i}^{n+1-i} (-1)^\ell \binom{n+1}{\ell+i} q_M(m-\ell) -
    \sum_{\ell=-i}^{-1} (-1)^\ell \binom{n+1}{\ell+i} q_M(m-\ell) \\
    &= \nabla^{n+1} q_M(m+i) - \sum_{\ell=1}^{i} (-1)^\ell
    \binom{n+1}{i-\ell} q_M(m+\ell) \\
    &= - \sum_{\ell=1}^{i} (-1)^\ell \binom{n+1}{i-\ell} h_M(m+\ell)
  \end{align*}
  which establishes the second family of inequalities.  Because the
  inequalities are equalities for an appropriate direct sum of the modules
  appearing in the list \eqref{eq:cyclic}, we conclude that the bound is sharp
  for some positive multiple of the Hilbert function $h_M$.
\end{proof}

We end by illustrating the final proposition in an example.

\begin{example}
  Let $n = 3$ and let $M$ be an $S$-module generated in degree zero and
  satisfying $h_M(j) = 3 j+1$ for all $j \in \NN$.  If the
  Castelnuovo--Mumford regularity of $M$ is bounded by $1$ or $2$
  respectively, then Proposition~\ref{pro:BettiBounds} produces the following
  entrywise bounds on the Betti tables:
  \begin{xalignat*}{2}
    \renewcommand{\arraystretch}{1.2}
    \renewcommand{\arraycolsep}{5pt} 
    &\begin{array}{c|ccccc}
      & 0 & 1 & 2 & 3 & 4 \\ \hline
      0 & 1 & . & . & . & . \\
      1 & . & 3 & 2 & . & . \\
    \end{array}
    &
    \renewcommand{\arraystretch}{1.2}
    \renewcommand{\arraycolsep}{5pt} 
    \begin{array}{c|ccccc}
      & 0 & 1 & 2 & 3 & 4 \\ \hline
      0 & 1 & . & . & . & . \\
      1 & . & 3 & 6 & \tfrac{9}{2} & \tfrac{6}{5}  \\
      2 & . & 4 & \tfrac{9}{2} & \tfrac{6}{5} & .  \\
    \end{array} 
  \end{xalignat*}
  
  \vspace{-1.2em} \hfill $\diamond$
\end{example}

\raggedright

\end{document}